\documentclass{amsart}

\usepackage{amssymb,amsthm,amscd,graphicx}

\newtheorem{thm}{Theorem}[section]
\newtheorem{prop}[thm]{Proposition}
\newtheorem*{thm*}{Theorem}
\newtheorem{defn}[thm]{Definition}
\newtheorem{rmk}[thm]{Remark}

\newtheorem{lem}[thm]{Lemma}
\newtheorem*{cor*}{Corollary}
\newcommand{\inlcomb}[2]{\left(\begin{smallmatrix} #1 \\ #2 \end{smallmatrix}\right)}
\newcommand{\sqcomb}[2]{\left[\!\!\! \begin{array}{c} #1 \\ #2 \end{array} \!\!\!\right]}
\newcommand{\inlsqcomb}[2]{\left[\begin{smallmatrix} #1 \\ #2 \end{smallmatrix}\right]}

\numberwithin{equation}{section}

\newcommand{\cB}{\mathcal{B}}
\newcommand{\cS}{\mathcal{S}}

\newcommand{\NN}{\mathbb{N}}
\newcommand{\id}{\textsf{id}}

\title[Cardinality of Balls]{Cardinality of Balls in Permutation Spaces}
\date{February 28, 2013}

\author[L.P Dinu]{Liviu P. Dinu}
\address{Faculty of Mathematics and Computer Science, University of Bucharest}
\email{ldinu@fmi.unibuc.ro}

\author[C.Zara]{Catalin Zara}
\address{Department of Mathematics, UMass Boston}
\email{catalin.zara@umb.edu}

\begin{document}

\begin{abstract}
For a right invariant distance on a permutation space $S_n$ we give a sufficient condition 
for the cardinality of a ball of radius $R$ to grow polynomially in $n$ for fixed $R$. 
For the distance $\ell_1$ we show that for an integer $k$ the cardinality of a sphere 
of radius $2k$ in $S_n$ (for $n \geqslant k$) is a polynomial of degree $k$ in $n$ and 
determine the high degree terms of this polynomial.
\end{abstract}

\maketitle

\tableofcontents

\section{Introduction}
Let $S_{\infty}$ be the group of permutations of $\NN = \{1, 2, \ldots, n, \ldots\}$ 
that fix all but finitely many entries. For each positive integer $n$ the group $S_n$ 
of permutations of $[n]=\{1, 2, \ldots, n\}$ can be naturally identified with the subgroup 
of $S_{\infty}$ containing the permutations that fix all the entries greater than $n$. 
With these identifications $S_n \subset S_m \subset S_{\infty}$ for all $0 < n < m$.
Let $D\colon S_{\infty} \times S_{\infty} \to [0,\infty)$ be a metric on $S_{\infty}$ for which 
right multiplications are isometries: 
\begin{equation*}
D(uw, vw) = D(u,v)
\end{equation*}
for all $u, v, w \in S_{\infty}$. Such metrics are called \emph{right invariant}; several examples are 
discussed in the survey \cite{dez}. Since permutations in $S_{\infty}$ fix all but finitely many 
entries, the following right invariant metrics are well-defined on $S_{\infty}$:
\begin{itemize}
\item The Hamming distance $H$, counting the number of disagreements, 
\begin{equation*}
H(u,v) = \#\{i\geqslant 1\; | \; u(i) \neq v(i)\}\; .
\end{equation*}
\item For $p \geqslant 1$, the $\ell_p$ distance
\begin{equation*}
\ell_p(u,v) = \biggl( \sum_{i \geqslant 1} |u(i) - v(i)|^p \biggr)^{1/p}\; ;
\end{equation*}
in particular
\begin{equation*}
\ell_1(u,v) = \sum_{i \geqslant 1} |u(i)-v(i)|\; .
\end{equation*}
\item The $\ell_{\infty}$ distance
\begin{equation*}
\ell_{\infty}(u,v) = \max_{i \geqslant 1} |u(i)-v(i)|\; .
\end{equation*}
\item The Cayley distance $T$, with $T(u,v)$ defined as the minimal number of transpositions needed to transform $u$ into $v$.
\item The Kendall distance $I$, with $I(u,v)$ defined as the minimal number of  transpositions of adjacent entries needed to transform $u$ into $v$. 
\end{itemize}

The restriction of a right invariant metric $D$ to $S_n$ is a right invariant metric on $S_n$ and
an immediate consequence of right invariance is that the cardinality of spheres 
and of balls depends only on the radius and is independent of the center.
For a permutation $u$ we define $D(u) = D(\id, u)$, where $\id$ is the identity permutation. For $R>0$, let 
\begin{equation*}
\mathcal{B}_D(R) = \{ u \in S_{\infty} \; | \; D(u) \leqslant R\}
\end{equation*}
be the closed ball of radius $R$ centered at $\id$ and
\begin{equation*}
\mathcal{S}_D(R) = \{ u \in S_{\infty} \; | \; D(u) = R\}
\end{equation*}
the sphere of radius $R$ centered at $\id$. There are at most countably many values of $R$ for which the  sphere $\cS_D(R)$ is not empty.

For $n \geqslant 1$, let 
\begin{equation*}
\cB_{D,n}(R) = \cB_D(R) \cap S_n =  \{ u \in S_{n} \; | \; D(u) \leqslant R\}
\end{equation*}
and 
\begin{equation*}
\cS_{D,n}(R) = \cS_D(R) \cap S_n =  \{ u \in S_{n} \; | \; D(u) = R\}
\end{equation*}
be the corresponding ball and sphere in $S_n$, and
\begin{equation*}
V_{D,n}(R) = \#\cB_{D,n}(R)\quad , \quad A_{D,n}(R) = \# \cS_{D,n}(R)
\end{equation*}
the cardinalities of a ball and the sphere of radius $R$ in $S_n$.

Formulas for $V_{D,n}(R)$ and $A_{D,n}(R)$ (approximate or exact) are known for several 
of the metrics mentioned above. For the Kendall distance $I$, the cardinality of the 
sphere  $\cS_{I,n}(k)$ is polynomial of degree $k$ in $n$ (see \cite[p. 15]{tacp3}). 
For the Hamming distance $H$, the cardinality of the sphere  $\cS_{H,n}(k)$ is also 
polynomial of degree $k$ in $n$; more precisely
\begin{equation*}
A_{H,n}(k) =\left\lfloor \frac{k!}{e} \right\rfloor \binom{n}{k}  \; ,
\end{equation*}
since $\left\lfloor \frac{k!}{e} \right\rfloor$ is the number of permutations in $S_k$ 
with no fixed points.

For both metrics $I$ and $H$, the cardinality of a ball of fixed radius $k$ is also a polynomial 
of degree $k$ in $n$. However, for the $\ell_{\infty}$ metric, the situation is different. A consequence 
of \cite[Proposition 4.7.8]{sta-ec1} is that for fixed $k$, the cardinality of a ball of $\ell_{\infty}$-radius $k$ in $S_n$ satisfies a linear recurrence in $n$. 
A lower bound, exponential in $n$, is given in \cite{klo}. 

In this article we give a sufficient condition for the cardinality of a ball of radius $R$ in $S_n$ 
to grow polynomially in $n$ for fixed $R$. The condition is satisfied 
by the Kendall distance $I$ and by the metrics $\ell_p$ with $p \geqslant 1$, 
but not by $\ell_{\infty}$ or by the Hamming distance $H$. 

For the metric $\ell_1$ we show that for an integer $k$, 
$A_{\ell_1, n}(2k)$ is a polynomial $P_k(n)$ of degree $k$ in $n$ (for $n \geqslant k$) and give explicit 
formulas for high degree terms of this polynomial.  A direct consequence of the polynomial growth is that, if we 
randomly pick a permutation in a closed ball of radius $2k$ then,  with probability converging to 1 as $n \to \infty$, 
the permutation is on the sphere of radius $2k$ with the same center:
\begin{equation*}
\lim_{n \to \infty} 
P\Bigl(\ell_1(u)=2k \; | \; \ell_1(u) \leqslant 2k\Bigr) = 1\; .
\end{equation*}

In Section~\ref{sec:split_types} we introduce the \emph{split types}, 
one of the main tools used to prove our results. 
In Section~\ref{sec:poly_growth} we formulate and prove 
the sufficient conditions for the cardinalities of balls and of spheres to grow polynomially, and 
in Section~\ref{sec:lp} we apply these conditions to the metrics $\ell_p$. 
In Section~\ref{sec:pol_form} we compute high degree terms of the polynomials $P_k(n) = A_{\ell_1,n}(k)$. 
Exact formulas for $\ell_1$-spheres of small radii are given in Section~\ref{sec:small_radii}.

\section{Split Types}
\label{sec:split_types}

We use the one-line notation for permutations: 
a permutation $\pi \in S_m$ is denoted by  $\pi = \ (\pi(1)\pi(2) \ldots \pi(m))$. 

\begin{defn}{\rm If 
$\pi \in S_m$ and $\sigma \in S_n$, the \emph{concatenation} $\pi+\sigma$ 
is the permutation 
\begin{equation*}
\pi + \sigma = (\pi(1)\ldots \pi(m) (m+\sigma(1)) \ldots (m + \sigma(n))) \in S_{m+n}\; .
\end{equation*}
}
\end{defn}
For example, $(321)+(21) = (32154)$. 

Concatenation is associative (but not commutative) so we can define the 
concatenation of any finite number of permutations. For example,
\begin{equation*}
(1) + (321) + (12) + (21) = (1432) + (1243) = (14325687)\; .
\end{equation*}

\begin{defn}{\rm  A permutation that can not be written as a concatenation of 
permutations of lower order is called \emph{connected}. 
A permutation that can be written as a 
concatenation of permutations of lower order is called 
\emph{disconnected}. The unique decomposition of a permutation as a concatenation of 
connected permutations is called the \emph{split decomposition} of that permutation.
\footnote{The connected parts of the split decomposition correspond 
to the components of the cycle diagram defined in \cite{eli}.}
}
\end{defn}
For example, $(1432)=(1) + (321)$ is disconnected, but $(321)$ is connected.

An easy way to determine the split decomposition 
is to identify the \emph{cuts}: a permutation $\pi \in S_n$ has a cut 
at $1\leqslant i < n$ if $\pi(j) \leqslant i$ for all $j \leqslant i$ 
and $\pi(j) >i$ for all $j > i$. If $1 \leqslant i_1 < i_2 < \dotsb < i_q < n$ 
are the cuts of $\pi$, then the slices $\pi_0 = \pi([1,i_1))$, 
$\pi_1= \pi([i_1, i_2))$, \ldots, $\pi_q= \pi([i_q, n])$ correspond to 
the permutations in the split decomposition of $\pi$. For example, 
for $\pi = (1432576) \in S_7$ the splits are at 1, 4, and 5 and the 
split decomposition is
\begin{equation*}
(1432576) = (1) + (321) + (1) + (21)\; .
\end{equation*}

\begin{defn}{\rm
The \emph{split type} of a permutation is the permutation obtained by
concatenating the nontrivial permutations of its split decomposition.
}
\end{defn}

For example, the split type of $(14325687) = (1) + (321) + (1)+(1) + (21) \in  S_8$ is 
\begin{equation*}
(321)+(21) = (32154) \in S_5 \; .
\end{equation*}

The necessary and sufficient condition that a permutation $\sigma$ be a split type 
is that the connected parts in the split decomposition 
of $\sigma$ be non-trivial permutations. Let $S_{\infty}'$ be the set of split types and $S_m'$ 
the subset of split types in $S_m$.
If $\sigma \in S_{\infty}'$, we define $q(\sigma)$ 
to be the number of connected parts, and $m(\sigma)$ to be the smallest $m$ such that 
$\sigma \in S_m$. Equivalently, $m(\sigma)$ is the unique value of $m$ such that $\sigma \in S_m'$. 
For example, if $\sigma = (32154)$, then $q(\sigma) = 2$ and $m(\sigma) = 5$. We have $m(\sigma) \geqslant 2q(\sigma)$, or, equivalently,
\begin{equation*}
q(\sigma) \leqslant m(\sigma) - q(\sigma)\; ,
\end{equation*}
with equality only for $\sigma = (21)+(21) + \dotsb + (21)$.

\begin{lem}\label{lem:Mnsigma}
{\rm
Let $\sigma \in S_\infty'$ be a split type. The number of permutations 
$\pi \in S_n$ of split type $\sigma$ is
\begin{equation*}
M(n,\sigma) = \sqcomb{n+q(\sigma)-m(\sigma)}{q(\sigma)} \; ,
\end{equation*}
where $\inlsqcomb{i}{j}$ is the binomial 
coefficient $\inlcomb{i}{j}$ if $0 \leqslant i \leqslant j$
and 0 otherwise.\footnote{If $i < 0 \leqslant  j$, then the binomial coefficient 
$\binom{i}{j}$ is $i(i-1)\ldots (i-j+1)/j! = (-1)^j \inlcomb{j-i-1}{j}$ and is not zero. 
However, $\inlsqcomb{i}{j} = 0$.}
}
\end{lem}

\begin{proof}
Let $\pi \in S_n$ of split type $\sigma = \sigma_1 + \dotsb + \sigma_q \in S_m$. 
We mark the leftmost position of each of the images in $\pi$ of the parts 
$\sigma_1$, \ldots $\sigma_q$, delete the other $m-q$ positions of those images, 
and then compress the result, translating the markings. The marked values form a 
$q$-element subset of $[n\!+\!q\!-\!m]$.
For example, for the permutation $\pi = (1432576)$ of split type
$\sigma = (321)+(21)$, the image of $\sigma$ in $\pi$ is on positions 2,3,4 and 6,7. 
We keep and mark (in bold) 2 and 6 and delete 3,4, and 7. 
The result is the set $\{1, \mathbf{2}, 5, \mathbf{6}\}$, which is 
compressed to $\{1, \mathbf{2}, 3, \mathbf{4}\}$, corresponding to 
the subset $\{2, 4\}$ of $[4]$.

We have therefore associated a $q$-element subset of $[n+q-m]$ to each permutation 
$\pi \in S_n$ of split type $\sigma$, and this map is clearly injective. It is also surjective, 
because the process is reversible: each $q-$element subset of $[n+q-m]$ can be 
expanded to a unique permutation in $S_n$ of split type $\sigma$. 
For example, if $\sigma = (321)+(21)$, then the subset $\{1,3\}$ of $[4]$ comes from $\{\mathbf{1}, 2,\mathbf{3}, 4\}$, 
which is the result of compressing $\{\mathbf{1}, 4, \mathbf{5}, 7\}$. That comes from 
the permutation $(3214657)$. 

Hence the number of permutations in $S_n$ of split  type 
$\sigma = \sigma_1 + \dotsb + \sigma_q \in S_m$ is the same as the number of 
$q$-element subsets of $[n+q-m]$.
\end{proof}
\begin{rmk}{\rm
For $n \geqslant m(\sigma)-q(\sigma)$, $M(n,\sigma)$ is a polynomial of degree $q(\sigma)$ in $n$.
}\end{rmk}

For all distances $D=H, T, I, \ell_p, \ell_{\infty}$, the split type determines the distance to the identity:
if a permutation $u \in S_n$ has split type $\sigma \in S_m$, then 
\begin{equation*}
D(u) = D(\sigma)\; .
\end{equation*}

\begin{defn}\label{def:split_dist}{\rm
A distance $D$ on $S_{\infty}$ is called a \emph{split type distance} if $D(u) = D(v)$ 
for any two permutations $u$ and $v$ having the same split type.
}\end{defn}

Moreover, if $D=I, \ell_1$, then $D(u+v) = D(u) + D(v)$ for all $u,v$. 

\begin{defn}\label{def:add_dist}{\rm
A distance $D$ on $S_{\infty}$ is called a \emph{additive} if $D(u+v) = D(v)+D(v)$.
}\end{defn}

Any additive distance is a split type distance.

\section{Polynomial Growth}
\label{sec:poly_growth}

We can now formulate and prove our sufficient condition for a polynomial growth of 
$V_{D,n}(R)$, the cardinality of  a $D$-ball of radius $R$ in $S_n$.

\begin{thm}\label{th:polygrowth}{\rm
Let $D$ be a right invariant split type distance on $S_{\infty}$ and $R>0$. Suppose there exists an 
integer $N(R)$ such that $m(\sigma) - q(\sigma) \leqslant N(R)$ for every split 
type $\sigma$ with $D(\sigma) \leqslant R$. Then there exists a polynomial $P$
of degree at most $N(R)$ such that $V_{D,n}(R) = P(n)$ for all $n \geqslant N(R)$. 
}\end{thm}

\begin{proof}
If $\sigma$ is a split type with $D(\sigma) \leqslant R$, then 
$q(\sigma) \leqslant m(\sigma) - q(\sigma) \leqslant N(R)$ and 
$m(\sigma) \leqslant N(R)+q(\sigma) \leqslant 2N(R)$. Then
\begin{equation*}
V_{D,n}(R) =  \sum_{\sigma \in S_{\infty}'} \; 
\sum_{D(\sigma) \leqslant  R} 
\sqcomb{n+q(\sigma)-m(\sigma)}{q(\sigma)} = \sum_{m \leqslant 2N(R)}\, 
\sum_{\substack{\sigma \in S_m' \\ D(\sigma) \leqslant R}} \sqcomb{n+q(\sigma)-m}{q(\sigma)} \; .
\end{equation*}
Let $\alpha_D(R,m,q)$ be the number of split types $\sigma \in S_m'$ 
with $q(\sigma) = q$ and $D(\sigma) \leqslant R$. Then 
\begin{equation*}
V_{D,n}(R) = \sum_{q=1}^{N(R)} \; \sum_{m=2q}^{q+N(R)} \alpha_D(R,m,q) 
\sqcomb{n+q-m}{q} 
\end{equation*}
and for $n \geqslant N(R)$, 
\begin{equation*}
V_{D,n}(R) = \sum_{q=1}^{N(R)} \; \sum_{m=2q}^{q+N(R)} \alpha_D(R,m,q) 
\binom{n+q-m}{q} \; 
\end{equation*}
is a polynomial in $n$ of degree at most $N(R)$.
\end{proof}

A completely similar argument proves the following condition for a polynomial 
growth of the cardinality of a sphere.

\begin{thm}\label{th:growthA}{\rm
Let $D$ be a right invariant split type distance on $S_{\infty}$ and $R>0$. Suppose there 
exists an integer $N(R)$ such that $m(\sigma) - q(\sigma) \leqslant N(R)$ for every split 
type $\sigma$ with $D(\sigma) = R$. Let $\beta_D(R,m,q)$ be the number of split types 
$\sigma \in S_m'$ such that $D(\sigma) = R$ and $q(\sigma) = q$. Then 
\begin{equation*}
A_{D,n}(R) = \sum_{q=1}^{N(R)} \; \sum_{m=2q}^{q+N(R)} \beta_D(R,m,q) 
\sqcomb{n+q-m}{q} 
\end{equation*}
and for $n \geqslant N(R)$, 
\begin{equation}\label{eq:polySphereD}
A_{D,n}(R) = \sum_{q=1}^{N(R)} \; \sum_{m=2q}^{q+N(R)} \beta_D(R,m,q) 
\binom{n+q-m}{q} \; 
\end{equation}
is a polynomial in $n$ of degree at most $N(R)$.
}\end{thm}

To determine the coefficient $\beta_D(R, m, q)$ by brute force, 
one would need to consider all permutations in $S_{2N(R)}$. 
The situation is better for additive distances.

\begin{defn}{\rm
Let $k$ be a positive integer. By a \emph{composition} of $k$ we mean an 
expression of $k$ as an ordered sum of positive integers, and by a \emph{$q-$composition} 
we mean a composition that has exactly $q$ terms. (See \cite{sta-ec1}.)
}\end{defn}

It is easy to see that the number of $q$-compositions 
of $k$ is $\binom{k-1}{q-1}$: a $q$-composition of $k$ is a way to break an array
of $k$ blocks into $q$ non empty parts. The first part must start at 1, but the remaining 
$q\!-\!1$ parts can start at any $q\!-\!1$ places in the remaining $k\!-\!1$ blocks.

\begin{prop}{\rm
Let $D \colon S_{\infty} \to \NN\cup \{0\}$ be an additive right invariant distance on $S_{\infty}$. Then 
\begin{equation}\label{eq:betagen}
\beta_D(k,m,q)= \sum_{(k_1, \ldots, k_q)}  \sum_{(m_1, \ldots, m_q)} \prod_{i=1}^q \beta_D(k_i,m_i,1) \; ,
\end{equation}
where the sums are over of $q-$compositions $(k_1, \ldots, k_q)$ of $k$
and $q-$compositions $(m_1, \ldots, m_q)$ of $m$.  
}\end{prop}

If the metric $D$ is additive, then all the coefficients $\beta_D(k,m,q)$ 
can be computed from $\beta(k',m',1)$ for all $k' \leqslant k$ and $m' \leqslant m$, and that can be done inside $S_{k+1}$.

\section{Polynomial Growth for $\ell_p$}
\label{sec:lp}

In this section we prove that the metrics $\ell_p$ for $p \geqslant 1$ and the Kendall metric $I$ satisfy 
the condition of Theorems~\ref{th:polygrowth} and \ref{th:growthA}, hence the spheres and the balls grow polynomially in $n$ for a fixed radius.

\begin{prop}\label{prop:poly_lp}{\rm
If $\sigma$ is a split type and $\ell_p(\sigma) \leqslant R$, then 
\begin{equation*}
m(\sigma) - q(\sigma) \leqslant \frac{1}{2} \,R^p \; .
\end{equation*}
}\end{prop}

\begin{proof}
If $\ell_p(\sigma) \leqslant R$, then
\begin{equation*}
\ell_1(\sigma) \leqslant \ell_p(\sigma)^p \leqslant R^p  \; ,
\end{equation*}
and therefore it suffices to prove the statement for the case $p=1$. 

Suppose that $\sigma \in S_m$ is a split type with split decomposition 
$\sigma = \sigma_1 + \dotsb + \sigma_q$ and such that $\ell_1(\sigma) \leqslant r$. 
We need  to show that $m-q \leqslant r/2$. It suffices to prove that for every 
\emph{connected} split type $\sigma$ we have 
\begin{equation}\label{eq:boundsm}
m(\sigma) -1 \leqslant \frac{1}{2} \, \ell_1(\sigma)\; .
\end{equation}

We prove \eqref{eq:boundsm} by induction on the number of cycles of $\sigma$.

Suppose $\sigma$ is a cycle $c \colon i_1 = 1 \to i_2 \to \dotsb \to i_t = m \to \dotsb \to i_s$.
Then
\begin{align}
\ell_1(\sigma) = &  (|i_1-i_2| + \dotsb + |i_{t-1}-i_t|) + 
(|i_t-i_{t+1}| + \dotsb + |i_{s}-i_1|) \geqslant \nonumber  \\
\geqslant & 2(i_t-i_1) = 2(m-1) \; ,\nonumber
\end{align}
and that implies \eqref{eq:boundsm}. The equality occurs if and only if the cycle $c$ is \emph{monotone}, 
\begin{equation*}
1=i_1 < i_2 < \dotsb < i_{t-1} < i_t =m > i_{t+1} > \dotsb > i_{s} > i_1=1\; ,
\end{equation*}
\emph{i.e.} the values occurring in $c$ between $i_1=1$ and $i_t=m$ appear 
in increasing order, and the values from $i_t=m$ to $i_1=1$ appear in decreasing order. 

Suppose now that \eqref{eq:boundsm} is valid for all connected permutations with at most 
$r$ cycles, and let  $\sigma$ be a connected permutation with $r+1$ cycles $c_1, \ldots, c_{r+1}$, 
ordered in increasing order of their minimal elements. 
Let $a=\min(c_{r+1})$, $b=\max(c_{r+1})$, and
$\sigma'=(c_1, \ldots, c_r)$, $a'=\min\{i\; | \; \sigma'(i) \neq i\}$, 
$b'=\max\{ i \; | \; \sigma'(i) \neq i\}$. Then $m = \max(b,b') > \min(b,b')> \max(a,a') > \min(a,a') = 1$.
Moreover, the split type of $\sigma'$ is a connected permutation in $S_{m'}$, where $m'=b'-a'+1$.

Using the induction hypothesis we obtain
\begin{align*}
\ell_1(\sigma) = & \ell_1(\sigma')+\ell_1(c_{r+1}) \geqslant 2(m'-1)+2(b-a) >  \\
> &  2(\max(b,b')-\min(a,a')) = 2(m-1) \; ,
\end{align*}
so the inequality \eqref{eq:boundsm} follows by induction. 
\end{proof}

Theorem~\ref{th:polygrowth} implies that the cardinality $V_{\ell_p,n}(R)$ 
of a ball of fixed radius $R$ with respect to the distance $\ell_p$  grows polynomially in $n$, with 
the degree of the polynomial not exceeding $R^p/2$. For the distance $\ell_1$ we can be more precise.

\begin{thm}{\rm
Let $k$ be a positive integer.  For $n \geqslant k$, we have
\begin{equation*}
A_{\ell_1,n}(2k) = \binom{n-k}{k} + \text{positive terms of degree at most } k-1 \; ,\; 
\end{equation*}
\begin{equation*}
V_{\ell_1,n}(2k) = \binom{n-k}{k} + \text{positive terms of degree at most } k-1\; .
\end{equation*}
}\end{thm}

\begin{proof}
If $\sigma$ is a split type with $\ell_1(\sigma)=2k$, then 
$q(\sigma) \leqslant m(\sigma) - q(\sigma) \leqslant k$. The split type 
$\sigma = (21) + \dotsb +(21)$ is the only one with $k$ parts, and $m(\sigma) = 2k$. 
Therefore $\beta_{\ell_1}(2k,m,k) = 1$ if $m=2k$ and 0 otherwise. Theorem~\ref{th:growthA} 
implies the formula for $A_{\ell_1,n}(2k)$, and the formula for $V_{\ell_1,n}(2k)$ is an 
immediate consequence.
\end{proof}

In the rest of this section we discuss results related to other distances.

\begin{rmk}{\rm
If $I$ is the Kendall metric, then $A_{I,n}(k)$ is a polynomial of degree $k$ in $n$ for $n \geqslant k$ 
(see \cite[p. 15]{tacp3}, where an exact formula is also given). We use our result on $\ell_1$ 
to give an alternative proof of this result. By \cite[Theorem~2]{dia-gra} we know 
that $\ell_1(u) \leqslant 2I(u)$, and therefore
\begin{equation*}
S_{I,n}(k) \subset B_{I,n}(k) \subseteq B_{\ell_1,n}(2k)\; .
\end{equation*}
As a consequence, both $A_{I,n}(k)$ and $V_{I,n}(k)$ grow at most polynomially of 
degree $k$ in $n$. If $\sigma$ is a split type and $I(\sigma) \leqslant k$, then 
Proposition~\ref{prop:poly_lp} implies $m(\sigma) - q(\sigma) \leqslant k$; therefore
$q(\sigma) \leqslant k$ and the polynomial growth formula \eqref{eq:polySphereD} is 
valid for $n \geqslant k$. There exists exactly one split type 
$\sigma$ with $I(\sigma) = k$ and $q(\sigma) = k$. This is the split type
$\sigma = (21) + (21) + \dotsb + (21)\; ,$
with $k$ terms in the concatenation, and $m(\sigma) = 2k$. Hence $\beta_I(k, m, k) = 1$ if 
$m=2k$ and 0 otherwise. Therefore for $n \geqslant k$,
\begin{equation*}
A_{I,n}(k) = \binom{n-k}{k} + \text{positive terms of degree at most } k-1 \; ,\; 
\end{equation*}
\begin{equation*}
V_{I,n}(k) = \binom{n-k}{k} + \text{positive terms of degree at most } k-1\; .
\end{equation*}
Hence $A_{I,n}(k) \simeq A_{\ell_1,n}(2k)$ and $V_{I,n}(k) \simeq V_{\ell_1,n}(2k)$ 
up to terms of degree $k\!-\!1$, and therefore the relative errors converge to 0 as $n$ 
goes to $\infty$.

Even though Equation~\ref{eq:polySphereD} does not seem to give a general formula for 
$A_{I,n}(k)$ as in \cite{tacp3}, it  has the advantage that the formulas it generates are 
\emph{positive}: they express $A_{I,n}(k)$ as a sum of positive terms, consistent 
with the fact that $A_{I,n}(k)$ \emph{counts} permutations with certain properties.}
\end{rmk}

\begin{rmk}{\rm
The cardinalities of $\ell_{\infty}$-balls of fixed radii satisfy linear recurrences in $n$ (\cite[Prop. 4.7.8]{sta-ec1}, \cite{kr-sc}). They grow at least 
exponentially in $n$ and an exponential lower bound is also given in \cite{klo}. 
Note that the $\ell_{\infty}$ distance does not satisfy the conditions of Theorems~\ref{th:polygrowth} and \ref{th:growthA}: if $\sigma_r$ is the concatenation of $r$ copies of $(21)$, then $\ell_{\infty}(\sigma_r) = 1$ but $m(\sigma_r) - q(\sigma_r) = r \to \infty$ as $r\to \infty$.
}\end{rmk}

\begin{rmk}{\rm
The Hamming distance shows that the condition is not necessary. Cardinalities of spheres and balls of fixed radii grow polynomially, but for every $r$, the transposition $\sigma$ that 
swaps $1$ 
and $r$ is a split 
type in $S_r$ and $H(\sigma_r) = 2$, even though $m(\sigma_r) - q(\sigma_r) = r-1 \to \infty$ 
as $r\to \infty$. What remains true, however, is that an eventual polynomial growth 
implies a bound on the number of parts of a split type: let $d(R)$ be the degree of $P$. 
The terms of highest degree in \eqref{eq:polySphereD} are not cancelled, hence that highest 
degree 
cannot exceed $d(R)$. Therefore $q(\sigma) \leqslant d(R)$ for all 
split types $\sigma$ with $D(\sigma) \leqslant R$.
}\end{rmk}

\section{High Degree Terms}
\label{sec:pol_form}

In the remaining sections of this paper we focus on the cardinality of spheres for the distance $D=\ell_1$, 
namely on $A_{\ell_1,n}(2k) = A_n(2k)$. For $n \geqslant k$, this number is given by the $k^{th}$-degree polynomial 
\begin{equation}\label{eq:P_k}
P_k(n) = 
\sum_{q=1}^{k} \; \sum_{m=2q}^{q+k} \beta(2k,m,q) 
\binom{n+q-m}{q} = \binom{n-k}{k} + \begin{array}{c} \text{ lower degree} \\ \text{terms}
\end{array} \; .
\end{equation}

In this section we determine high degree terms of this polynomial.  

Maximal distances in $S_n$ and cardinalities of spheres of maximal radius 
have been determined in \cite{dia-gra}. Their results imply the following.

\begin{lem}\label{lem:max_spheres}
{\rm If $m=2r$ is even, then $\beta(2k,2r,q)=0$ if $k>r^2$ and 
\begin{equation*}
\beta(2r^2, 2r, q) = \left\{ 
\begin{array}{cl} 
(r!)^2 , & \text{ if } q=1\\
0 , & \text{ otherwise}.
\end{array} \right.
\end{equation*}

If $m=2r+1$ is odd, then $\beta(2k, 2r+1,q) = 0$ if $k > r^2+r$ and 
\begin{equation*}
\beta(2r^2+2r, 2r+1, q) = \left\{ 
\begin{array}{cl} 
(2r+1)(r!)^2 , & \text{ if } q=1\\
0 , & \text{ otherwise}.
\end{array} \right.
\end{equation*}
}\end{lem}

\begin{proof}
If $m=2r$ is even, then the maximal distance $2k=2r^2+2r$ is 
achieved by permutations $\sigma \in S_{2r}$ with the property that 
$\sigma(i) >r$ for all $i \leqslant r$, and there are $(r!)^2$ such permutations. 
All of them are connected, because the cycle containing 1 must coincide or overlap 
with the cycle containing $m=2r$, hence all of them are split types with $q=1$.

If $m = 2r+1$ is odd, then the maximal distance is attained 
for permutations $\sigma \in S_{2r+1}$ with the property that $\sigma(i) >r$ 
for $i \leqslant r$ and $\sigma(i) \leqslant r$ for $i > r$. 
There are $(2r+1)(r!)^2$ such permutations. All of them are connected, 
because the cycle containing 1 must 
coincide or overlap with the cycle containing $m=2r+1$,
hence they are all split types with $q=1$.
\end{proof}

For small $m$ the maximal $\ell_1$-distances in $S_m$ are given below.

\begin{center}
\begin{tabular}{|c|c|c|c|c|c|c|}
\hline 
$m$ & 2 & 3 & 4 & 5 & 6 & 7 \\
\hline
\text{2k} & 2 & 4 & 8 & 12 & 18 & 24\\
\hline
\end{tabular}
\end{center}

Therefore:
\begin{itemize}
\item If $k \leqslant 3$ and $\beta(2k,m,1) \neq 0$, then $m = k+1$.
\item If $k \leqslant 5$ and $\beta(2k,m,1) \neq 0$, then $m \geqslant k$.
\item If $k \leqslant 7$ and $\beta(2k,m,1) \neq 0$, then $m \geqslant k-1$.
\item If $k \leqslant 8$ and $\beta(2k,m,1) \neq 0$, then $m \geqslant k-2$.
\item If $k \leqslant 9$ and $\beta(2k,m,1) \neq 0$, then $m \geqslant k-3$.
\end{itemize}

These simple observations allow us to prove the following result.

\begin{lem}\label{lem:boundestim}
{\rm  If $q \geqslant k-8$ and $\beta(2k,m,q) \neq 0$, 
then:
\begin{itemize}
\item either $m \geqslant k+q-3$, or
\item $q=k-8$, $m=2k-12$.  Then $\beta(2k, 2k-12, k-8) = 36(k-8)$, and $k \geqslant 9$. 
\end{itemize}
}
\end{lem}

\begin{proof}
By Equation~\eqref{eq:betagen} we have
$$\beta(2k,m,q) = \sum_{(k_1, \ldots, k_q)} 
\sum_{(m_1, \ldots, m_q)} \prod_{i=1}^q \beta(2k_i,m_i,1)$$
with sums over $q-$compositions of $k$ and $m$, respectively. 
Then $\beta(2k,m,q)~\neq~0$ implies that for at least one pair of compositions, 
all the terms $\beta(2k_i,m_i,1)$ are non-zero. Suppose $(k_1, \ldots, k_q)$ and 
$(m_1, \ldots, m_q)$ are compositions of 
$k$ and $m$ respectively such that $\beta(2k_i,m_i,1) \neq 0$ for all 
$i = 1, \ldots, q$. Without loss of generality we can assume 
$k_1 \geqslant k_2 \geqslant \dotsb \geqslant k_q$. Then 
$$k = k_1 + \dotsb + k_q \geqslant k_1+k_2+k_3 + (q-3) \geqslant k_1+k_2+k_3 + k-11\; ,$$
so $k_1 + k_2 + k_3 \leqslant 11$. Similarly $k_1 + k_2 \leqslant 10$ 
and $k_1 \leqslant 9$. Hence at most two parts of the composition 
have $k_i \geqslant 4$,  at most one part has $k_i \geqslant 6$,  and no part exceeds 9.
That shows that only the following situations are possible:
\begin{itemize}
\item $k_1 \leqslant 3$. No part is above 3, hence $m_i = k_i+1$ for all $q$ parts, 
so $m=k+q$.
\item $k_2 \leqslant 3 < k_1 \leqslant 8$. Only one part is above 3, and that 
part is under 8. Then $m_i =k_i+1$ for all but at most one part, 
and $m_1 \geqslant k_1 -2$ for the last part. In this case we have $m \geqslant k+q-3$.
\item $k_2 \leqslant 3 < k_1 =9$. Only one part is above 3, and that 
part is of length 9. Then $q=k-8$, $k_2 = \dotsb = k_q=1$, $m_2=\dotsb = m_q =2$ 
and $6 \leqslant m_1 \leqslant 10$, so $m = m_1 + 2(q-1) = m_1 + 2k -18$. The 
only case when $m < k+q-3 = 2k-11$ is when $m_1=6$, and then $m=2k-12$. Then  
the $q$-compositions of $k$ and $m$ are $(9,1, \ldots, 1)$ and $(6,2,\ldots, 2)$, 
respectively. There are $q=k-8$ possible places for the part of $k$ of 
length $9$,  $\beta(18,6,1) = 36$, and $\beta(2,2,1) =1$, hence 
$\beta(2k, 2k-12, k-8) = 36(k-8)$.
\item $k_3 \leqslant 3 < k_2 =4 \leqslant k_1 =6$. Only two parts are above 3: 
one is $k_1=6$ and the other is $k_2=4$. Then $m_i = k_i +1$ for all but at most 
two parts, and for those two parts $m_2 \geqslant k_2$ and $m_1 \geqslant k_1-1$. 
In this case we have $m \geqslant k+q-3$.
\item $k_3 \leqslant 3 < k_2 \leqslant k_1 \leqslant 5$. Then 
$m_1 \geqslant k_1$, $m_2 \geqslant k_2$, and $m_i \geqslant k_i+1$ for all other 
parts, hence $m\geqslant k+q-2$.
\end{itemize}

Therefore the only case when $m < k+q-3$ is when $q=k-8$ and $m=2k-12$. 
In that situation, $\beta(2k,2k-12, k-8) = 36(k-8)$, hence we must have 
$k \geqslant 9$.
\end{proof}

The following estimate is an immediate consequence of Lemma~\ref{lem:boundestim}.

\begin{lem}\label{th:estimate}{\rm
The $k^{th}$-degree polynomial
\begin{equation*}
Q_k(n) = \sum_{q=1}^k \sum_{m=k+q-3}^{k+q} \beta(2k,m,q) \binom{n-m+q}{q}   + 36(k\!-\!8)\binom{n\!-\!k\!+\!4}{k\!-\!8}  \; 
\end{equation*}
agrees with $P_k$ on terms of degree $k-8$ and higher. Moreover $P_k = Q_k$ for $k \leqslant 9$.
}\end{lem}

To determine $Q_k$ we need to determine $\beta(2k,m,q)$ for $m=k+q, \ldots,  k+q-3$.
We start with a technical lemma.

\begin{lem}\label{lem:maxcycles}
{\rm If $\sigma \in S_m$ is a connected split type and $\ell_1(\sigma) = 2k$, then 
$\sigma$ has at most $k+2-m$ cycles.
}\end{lem}

\begin{proof}
Let $c_1, \ldots, c_t$ be the cycles in $\sigma$, with $t\geqslant 1$.
Let $[i_1,j_1]$, \ldots, $[i_t,j_t]$ be the ranges of those cycles, such that $1=i_1\! <\! i_2\! <\! \dotsb\! < \!i_t$. 
Let $[j_1',\ldots, j_t']$ be the right endpoints of those ranges, 
sorted in increasing order: $\{j_1',\ldots, j_t' \} = \{j_1,\ldots,j_t\} $ and 
$j_1'\! <\! j_2'\! <\! \dotsb\! < \!j_t' = m$.
Since $\sigma$ is connected, its cycles are linked, hence $j_s' > i_{s+1}$ for all $s=1,\ldots, t\!-\!1$. 
Then
\begin{align}\label{eq:cyclebound}
2k  &=  \ell_1(\sigma) = \ell_1(c_1) + \dotsb + \ell_1(c_t) \geqslant 2(j_1-i_1) + \dotsb + 2(j_t-i_t) \nonumber= \\
 &=   2(m-1) + 2(j_1'-i_2) + \dotsb + 2(j_{t-1}'-i_t) \geqslant 2(m-1)+2(t-1)=  \\
 &=  2(m+t-2)\; . \nonumber
\end{align}
Hence $\sigma$ can't have more than $t_0 = k+2-m$ cycles. 
\end{proof}

Let $t$ be the number of cycles of a connected split type $\sigma \in S_m$. 
The excess of $\ell_1(\sigma)$ over the minimum $2(m+t-2)$ occurs from two sources, corresponding to the inequalities in the first two lines of \eqref{eq:cyclebound}:
\begin{enumerate}
\item cycles of $\sigma$ are not monotone, \emph{i.e.} values that occur in the 
part from $i_s$ to $j_s$ are not in increasing order, or values that occur in 
the part from $j_s$ to $i_s$ are not in decreasing order.
\item cycles of $\sigma$ are not minimally linked, \emph{i.e.} $j_s' - i_{s+1}>1$ 
for some endpoints;
\end{enumerate}

For example, the cycle $c=(1,4,2,3,6,5,10,8,9)$ has the blocks $4\to 2 \to 3$ and $6\to 5$
in the part from 1 to 10 but not in increasing order, and the block $8\to 9$ in the 
part from 10 to 1 but not in decreasing order.

We can now compute the coefficients $\beta(2k,m,q)$ for $m=k+q, \ldots,  k+q-3$.

\begin{lem}\label{lem:coeffs}
{\rm
The coefficients $\beta(2k,m,q)$ for $m=k+q, \ldots, k+q-3$ are
\begin{equation}\label{eq:k+q}
\beta(2k,k+q,q) = \binom{k\!-\!1}{q\!-\!1} 3^{k-q} 
\end{equation}
for all  $1 \leqslant q \leqslant k$;
\begin{equation}\label{eq:m=q+k-1}
\beta(2k,k\!+\!q\!-\!1,q) =  4(k\!-\!3) \, \binom{k\!-\!4}{q\!-\!1}\,3^{k-q-3}  
\end{equation}
if  $1 \leqslant q \leqslant k-3$, and 0 otherwise;
\begin{equation}\label{eq:m=q+k-2}
\beta(2k,k\!+\!q\!-\!2,q) =  4(k\!-\!5) \Bigl( (k\!-\!6)\binom{k\!-\!7}{q\!-\!1} + 15 \binom{k\!-\!6}{q\!-\!1} \Bigr) 3^{k-q-6}  \; 
\end{equation}
if $1 \leqslant q \leqslant k-5$, and 0 otherwise;
\begin{equation}\label{eq:m=q+k-3}
\beta(2k, k\!+\!q\!-\!3,q) = 4(k\!-\!7)\binom{k\!-\!8}{q\!-\!1} \bigl(8(k-q)^2 + 60(k-q) - 137 \bigr) 3^{k-q-10}\;
\end{equation}
if $1 \leqslant q \leqslant k-7$, and 0 otherwise.
}\end{lem}

\begin{proof}
To prove \eqref{eq:k+q} we start with $q=1$ and determine $\beta(2k,k+1,1)$ for $k \geqslant 1$.  
If $\sigma \in S_{k+1}$ is a connected 
split type counted by $\beta(2k, k+1, 1)$, then by Lemma~\ref{lem:maxcycles} 
$\sigma$ must be a cycle, 
and by the proof of Proposition~\ref{prop:poly_lp}, this cycle must be monotone: 
the values between 1 and $k+1$ must occur in increasing order, 
and the values between $k+1$ and $1$ in decreasing order. 
For each of the $k-1$ 
values $2, 3, \ldots, k$ we have three possibilities: the value appears on the 
part from 1 to $k+1$, on the part from $k+1$ to 1, or doesn't appear at all 
in the cycle. 
The choices are independent and each set of choices completely determines 
the permutation $\sigma$. Therefore
\begin{equation*}
\beta(2k, k+1, 1) = 3^{k-1}\; .
\end{equation*}

If $(k_1, \ldots, k_q)$ and $(m_1, \ldots, m_q)$ are $q$-compositions 
that appear in the sum \eqref{eq:betagen} and $m = k+q$, then $m_i = k_i+1$ for all $i=1,\ldots, q$ and therefore 
\begin{align*}
\beta(2k, k+q, q) & = \sum_{(k_1, \ldots, k_q)} \prod_{i=1}^q \beta(2k_i,k_i+1,1)  =
\sum_{(k_1, \ldots, k_q)} \prod_{i=1}^q 3^{k_i-1} = \\
& = 3^{k-q} \sum_{(k_1, \ldots, k_q)} 1 = 3^{k-q} \binom{k-1}{q-1}\; .
\end{align*}

To prove the formula \eqref{eq:m=q+k-1} we consider first the case $q=1$ and prove that 
\begin{equation}\label{eq:almmax}
\beta(2k,k,1) = 4(k-3)3^{k-4}\; .
\end{equation}

If $\sigma \in S_{k}$ is a split type such that $\ell_1(\sigma)=2k$, then $\sigma$ can 
have at most 2 cycles.

If $\sigma$ has two cycles $c_1$ and $c_2$, then they have to overlap, 
so $\text{range}(c_1) = [1,j]$ and $\text{range}(c_2) = [i,k]$,  
or $\text{range}(c_1) = [1,k]$ and $\text{range}(c_2) = [i,j]$, 
with $i<j$. Since 
\begin{equation*}
2k=\ell_1(\sigma)=\ell_1(c_1)+\ell_1(c_2) \geqslant 2k + 2(j-i-1) \geqslant 2k \; ,
\end{equation*}
we must have $j=i\!+\!1$ and the cycles have to be monotone. In both cases 
there are $k\!-\!3$ ways of choosing the pair $(i,i\!+\!1)$. For each choice, 
for each of the remaining $k\!-\!4$ values, the cycle it belongs to, if any, is determined. 
There are three possibilities for each of the remaining $k\!-\!4$ values:  
to be on the ascending part of the cycle, on the descending part, or 
to not be in the cycle at all. Consequently, there are $2(k\!-\!3) 3^{k-4}$ split types
with two cycles.

If $\sigma$ has only one cycle $c$, the only way to increase $\ell_1$ by exactly 2 over 
the minimum $2(k-1)$ is to have two adjacent values occurring in the opposite 
order to the ascending/descending part they belong to:
$1\! \to \!i\!+\!1\! \to\! i \!\to\! k$ or $k\!\to\! i \!\to\! i\!+\!1\! \to\! 1$.  In each 
of the two cases there are $k-3$ ways 
of choosing the pair $(i,i+1)$, and for each choice, there are three possibilities 
for each of the remaining $k-4$ values. Hence the number of split types $\sigma \in S_k$
consisting of one cycle such that $\ell_1(\sigma) = 2k$ is $2(k-3)3^{k-4}$. 

Therefore 
\begin{equation*}
\beta(2k,k,1)  = 2(k-3)3^{k-4} + 2(k-3)3^{k-4} = 4(k-3)3^{k-4}\; ,
\end{equation*}
hence \eqref{eq:almmax}. In particular, for $\beta(2k,k,1)$ to be positive, 
we must have $k\geqslant 4$.

Let $(k_1, k_2, \ldots, k_q)$ be a composition of $k$. 
If a composition $(m_1, \ldots, m_q)$ of $m=k+q-1$ contributes to 
$\beta(2k,k+q-1,q)$, then $m_i \leqslant k_i+1$ for all $i=1,\ldots, q$, which implies
\begin{equation*}
k+q-1 = m = m_1+\dotsb + m_q \leqslant (k_1+1)+\dotsb + (k_q+1) = k+q\; .
\end{equation*}
This can only happen is if $m_i=k_i+1$ for $q-1$ values of $i$ 
and $m_{i_0}=k_{i_0}$ for exactly one value $i_0$. But for 
$\beta(2k_{i_0},k_{i_0}, 1)$ to be nonzero, we must have $k_{i_0} \geqslant 4$, 
hence one of the parts of the composition of $k$ must be of order at least 4, which 
implies $q \leqslant k-3$.

We determine all the pairs $(k_1, \ldots, k_q), (m_1, \ldots, m_q)$ of 
compositions of $k$ and $m=q+k-1$ that give nonzero contributions to 
$\beta(2k,k+q-1,q)$ by deciding beforehand which index $i$ 
corresponds to the term with $m_{i}=k_{i}$. Select an index $1 \leqslant i_0 \leqslant q$. Let 
$[(k_1, \ldots, k_q), (m_1, \ldots, m_q)]$ be a pair of partitions with a 
nonzero contribution such that $m_{i_0}=k_{i_0}$. Then $m_i=k_i$ for all $i\neq i_0$, 
hence the partition $(m_1, \ldots, m_q)$ is completely determined by the pair 
$[(k_1, \ldots, k_q), i_0]$. Let $k_i'=k_i$ for all $i\neq i_0$ and 
$k_{i_0}' = k_{i_0}-3 \geqslant 1$. Then $(k_1', \ldots, k_q')$ is a 
composition of $k-3$. The process is reversible, hence the pairs 
$[(k_1, \ldots, k_q), (m_1, \ldots, m_q)]$ are indexed by pairs 
$[(k_1', \ldots, k_q'), i_0]$ consisting of a partition of $k-3$ into 
positive parts and an index $1\leqslant i_0 \leqslant q$. The contribution of such a pair is
\begin{equation*}
\beta(2k_1, \alpha_1+1, 1) \cdot \dotsb \beta(2k_{i_0}, \alpha_{i_0}, 1) \cdot \dotsb \cdot \beta(2k_q, \alpha_q+1, 1) =4k_{i_0}' 3^{k-q-3}\; .
\end{equation*}
The contribution of all pairs from the same composition of $k-3$ is then 
\begin{equation*}
4(k_1'+ \dotsb + k_q')3^{k-q-3} = 4(k-3) 3^{k-q-3}\; ,
\end{equation*}
and there are $\inlcomb{k-4}{q-1}$ such compositions, hence \eqref{eq:almmax}.

The proofs of \eqref {eq:m=q+k-2} and \eqref {eq:m=q+k-3} are similar but the 
bookkeeping is more elaborate.
\end{proof}

\begin{rmk}{\rm
The polynomial 
\begin{equation*}
R_k(n) = \sum_{q=1}^k \beta(2k,k+q,q) \binom{n-k}{q} = \sum_{q=1}^k \binom{k\!-\!1}{q\!-\!1} 3^{k-q} \binom{n-k}{q}\; 
\end{equation*}
agrees with $P_k$ on terms of degree $k\!-\!2$ and higher and is exact for $k \leqslant 3$. The generating function of $R_k$ is
\begin{equation*}
f_k(X) = \sum_{n\geqslant 0} R_k(n) X^n = \frac{X^{k+1}(2X-3)^{k-1}}{(X-1)^{k+1}}\; .
\end{equation*}
}\end{rmk}

\section{Spheres of Small Radius}
\label{sec:small_radii}
 
Combining Lemmas~\ref{th:estimate} and \ref{lem:coeffs} we obtain formulas for $P_k(n)$ for 
$k\leqslant \min(9,n)$. The polynomials for $k \leqslant 6$ are:
\begin{align*}
P_{1}(n) = & \binom{n\!-\!1}{1} = n-1 \\
P_{2}(n) = & \binom{n\!-\!2}{2} + 3\binom{n\!-\!2}{1} = \frac{1}{2}(n^2+n-6)  \\ 
P_3(n) =  & \binom{n\!-\!3}{3}+ 6\binom{n\!-\!3}{2}+ 9\binom{n\!-\!3}{1} = \frac{1}{6}(n^3+6n^2-25n-6)  \\
P_4(n) =  & \binom{n\!-\!4}{4} + 9\binom{n\!-\!4}{3}+ 27\binom{n\!-\!4}{2}+ 27\binom{n\!-\!4}{1}+ 4\binom{n\!-\!3}{1} \\
P_5(n) = & \binom{n\!-\!5}{5}+ 12\binom{n\!-\!5}{4}+ 54\binom{n\!-\!5}{3}+ 108\binom{n\!-\!5}{2}+ 81\binom{n\!-\!5}{1}+\\
& + 8\binom{n\!-\!4}{2}+ 24\binom{n\!-\!4}{1} 
    \\
P_6(n) = & \binom{n-6}{6} + 15 \binom{n-6}{5}+90\binom{n-6}{4}+270\binom{n-6}{3}+405\binom{n-6}{2} +\\
& + 243\binom{n-6}{1}+   12 \binom{n-5}{3}+ 240 \binom{n-5}{2} + 108\binom{n-5}{1} + 20 \binom{n-4}{1}\; .
\end{align*}

We can use the expressions above to compute the cardinality of spheres of 
radius $2k$ in $S_n$ for $k \leqslant 6$ and for \emph{all} $n\geqslant 2$, not just for $n \geqslant k$. 
All we have to do is replace 
the binomial coefficient $\inlcomb{a}{b}$ by 0 if $a<0$. Nothing interesting occurs 
for $k\leqslant 5$: $A_{\ell_1,2k}(n) = P_k(n)$ if $n \geqslant k$ and $A_{\ell_1,2k}(n)=0$, if $n<k$.  
But when $k=6$, then 
$A_{\ell_1,12}(n)=P_{6}(n)$ for $n\geqslant 6$,  and $A_{\ell_1,12}(n)=0$ only for $n\leqslant 4$. 
For $n=5$,
\begin{equation*}
A_{\ell_1,12}(5) = \# \{ u \in S_5\; |\; \ell_1(u) = 12\} = 20
\end{equation*} 
is computed from the expression for $P_6(5)$ by ignoring all but the 
last term.

\end{document}